\documentclass[12pt,a4paper]{amsart}
\usepackage{amsfonts}
\usepackage[active]{srcltx}
\usepackage[notref,notcite]{showkeys}
\usepackage{enumerate}

\usepackage{amsmath,amssymb,xspace,amsthm}

\newtheorem{theorem}{Theorem}
\newtheorem{lemma}[theorem]{Lemma}
\newtheorem{corollary}[theorem]{Corollary}
\numberwithin{equation}{section}

\def\LL{\mathcal{L}}
\def\Vir{{\mathfrak V}}

\def\span{\operatorname{span}}

\newcommand{\C}{\ensuremath{\mathbb C}\xspace}

\renewcommand{\a}{\ensuremath{\alpha}}
\renewcommand{\b}{\ensuremath{\beta}}

\newcommand{\Z}{\ensuremath{\mathbb{Z}}\xspace}
\newcommand{\N}{\ensuremath{\mathbb{N}}\xspace}

\newcommand{\R}{\ensuremath{\mathbb{R}}\xspace}

\renewcommand{\phi}{\varphi}

\begin{document}
\title[Non-weight Virasoro modules]{Irreducible Virasoro modules from irreducible Weyl modules}
\author{Rencai Lu and Kaiming Zhao}
\maketitle

\begin{abstract} We use Block's results to classify irreducible modules
over the differential operator algebra $\C[t,t^{-1}, \frac d{dt}]$.
 From modules $A$ over $\C[t,t^{-1}, \frac d{dt}]$ and using the ``twisting technique"
 we construct a class of modules $A_b$ over the Virasoro algebra for any $b\in \C$.
 These new   Virasoro
modules are generally not weight modules. The necessary
and sufficient conditions for $A_b$ to be irreducible are obtained. Then we determine the necessary
and sufficient conditions for two such irreducible Virasoro modules
to be isomorphic. Many interesting examples for such irreducible Virasoro
modules with different features are provided at the end of the
paper. In particular the class of irreducible Virasoro modules
$\Omega(\lambda, b)$ for any $\lambda\in\C^*$ and any $b\in\C$ are defined on the polynomial algebra $\C[x]$.
\end{abstract}

\vskip 10pt \noindent {\em Keywords:}  the Virasoro algebra, non-weight irreducible
module, Weyl module, twisting technique

\vskip 5pt \noindent {\em 2000  Math. Subj. Class.:} 17B10, 17B20,
17B65, 17B66, 17B68

\vskip 10pt

\section{Introduction}
We denote by $\mathbb{Z}$, $\mathbb{Z}_+$, $\N$, $\R$ and
$\mathbb{C}$ the sets of  all integers, nonnegative integers,
positive integers, real numbers and complex numbers, respectively.

Let $\Vir$ denote the complex {\em Virasoro algebra}, that is the
Lie algebra with basis $\{c, d_i:i\in\mathbb{Z}\}$ and the Lie
bracket defined (for $i,j\in\mathbb{Z}$) as follows:
\begin{displaymath}
[d_i,d_j]=(j-i) d_{i+j}+\delta_{i,-j}\frac{i^3-i}{12}{c};\quad
[d_i,{c}]=0.
\end{displaymath}
The algebra $\Vir$ is one of the most   important Lie algebras both
in mathematics and in mathematical physics, see for example
\cite{KR,IK} and references therein.
 The Virasoro algebra theory has been widely used in many physics areas
and other mathematical branches, for example, quantum physics [GO],
conformal field theory [FMS], Higher-dimensional WZW models [IKUX,
IKU], Kac-Moody algebras [K, MoP], vertex algebras [LL], and so on.

The representation theory of the Virasoro algebra has been
attracting a lot of attention from mathematicians and physicists.
There are two classical families of simple Harish-Chandra
$\Vir$-modules: highest weight modules (completely described in
\cite{FF}) and the so-called intermediate series modules. In
\cite{Mt} it is shown that these two families exhaust all simple
weight Harish-Chandra modules. In \cite{MZ1} it is even shown that
the above modules exhaust all simple weight modules admitting a
nonzero finite dimensional weight space.

Naturally, the next task is to study irreducible non-weight modules and irreducible weight
modules with infinite dimensional weight spaces.  Irreducible  weight
modules with infinite dimensional weight spaces were firstly constructed by taking  the tensor product of some
highest weight modules and some intermediate series modules in
\cite{Zh} in 1997, and the necessary and sufficient conditions for
such tensor product to be simple were recently obtained in
\cite{CGZ}. Conley and Martin gave another class of such examples
with four parameters in \cite{CM} in 2001. Then very
recently, new weight simple Virasoro modules were found
in \cite{LLZ, LZ}.  

Various  families of irreducible non-weight Virasoro modules were studied in
\cite{OW,LGZ,FJK,Ya,GLZ,OW2, MW}. These include  various versions of
Whittaker modules  constructed using different tricks. In
particular, all the above Whittaker modules and even more were
described in a uniform way in \cite{MZ2}.

The main purpose of the present paper is to construct new
irreducible (non-weight) Virasoro modules. Let us first introduce
the algebras we will use.

Let $\C[t]$ be the (associative) polynomial algebra. By $\partial$
we denote the operator $t\frac{{\rm d}}{{\rm d} t}$ on $\C[t]$. We
see that $\partial t^i=t^i( \partial+i)$. Then we have the
associative algebra $\mathcal {A}=\C[t,\partial]$ which is a proper
subalgebra of the rank $1$ Weyl algebra $\C[t,\frac{{\rm d}}{{\rm d}
t}]$. Note that $\mathcal {A}$ is the universal enveloping algebra
of the $2$-dimensional solvable Lie algebra $\mathfrak{a}_1=\C
d_0\oplus\C d_1$ subject to $[d_0,d_1]=d_1$. See \cite{Bl, MZ2}. Let $\mathcal {K}=\C[t,t^{-1}, \partial]$ be the Laurent
polynomial differential operator algebra.

 Let $B$ be an associative or Lie algebra over $\C$ and $C$ be a
subspace of $B$. A module $V$ over $B$ is called $C$-torsion if
there exists a nonzero $f\in C$ such that $fv=0$ for some nonzero
$v\in V$; otherwise $V$ is called $C$-torsion-free.

\

The paper is organized as follows. In Sect.2, we use Block's results
to classify irreducible modules over the differential operator
algebra $\C[t,t^{-1}, \frac d{dt}]$. In Sect.3, by twisting
irreducible modules $A$ over the associative algebra $ \mathcal {K}$
we construct a class of modules $A_b$ over the Virasoro algebra for any
$b\in \C$.
 These new   Virasoro
modules are generally not weight modules. The necessary
and sufficient conditions for $A_b$ to be irreducible are obtained (Theorem 9).  And we determine the necessary and
sufficient conditions for two such irreducible Virasoro modules to
be isomorphic (Theorem 12). In Sect.4, we recover some old Virasoro
modules and give concrete new examples from such irreducible
Virasoro modules with different feature.  In particular a class of
interesting irreducible Virasoro modules $\Omega(\lambda, b)$ for any $\lambda\in\C^*$ and any $b\in\C$ (see
Sect.4.3) are defined on the polynomial algebra $\C[x]$. We also
prove that these irreducible modules $A_b$ are not isomorphic to any other known irreducible
Virasoro modules  (Theorem 17).

We like to mention that, very recently, the irreducible Virasoro modules $\Omega(\lambda, b)$ were used to determine the
necessary and sufficient conditions for the modules constructed in \cite{MW} to be irreducible, see \cite{TZ, TZ2}.
\

\begin{section}{Irreducible modules over the associative algebra $\mathcal {K}$}\end{section}

In this section we will obtain a classification of irreducible
modules over  the associative algebra $\mathcal {K}$.

\

\begin{lemma}\label{G+} Let $V$ be any $\C[t]$-torsion-free
irreducible  module over the associative algebra $\mathcal {A}$.
Then $V$ can be extended into a module over the associative algebra
$\mathcal {K}=\C[t,t^{-1}, \frac d{dt}]$, i.e., the action of
$\mathcal {A}$ on $V$ is a restriction of an irreducible $\C[t,
t^{-1}]$-torsion-free $\mathcal {K}$-module.
\end{lemma}

\begin{proof}  Noting that $t\cdot V$ is a submodule of $V$, and the action of $t$
is injective, we know that the action of $t$ is bijective on $V$.
Then we have the action of $t^{-1}$ on $V$. Thus $V$ becomes a
$\C[t, t^{-1}]$-torsion-free $\mathcal {K}$-module.
\end{proof}

The  classification for all irreducible modules over $\mathcal {A}$
was given in \cite{Bl}.

\begin{lemma}\label{G}Let $\b$ be an irreducible element in the associative algebra $\C(t)[\partial]$. Then
 $$ {\mathcal {K}}/(\mathcal {K}\cap
 (\C(t)[\partial]\b))$$
is a $\C[t, t^{-1}]$-torsion-free irreducible module over the
associative algebra $\mathcal {K}$. Moreover any
$\C[t, t^{-1}]$-torsion-free irreducible module over the associative
algebra $\mathcal {K}$ can be obtained in this way.
\end{lemma}
\begin{proof}Note that $\C[t, t^{-1}]$ is $\partial$-simple (i.e., it has no nontrivial
$\partial$-invariant ideal). The lemma follows from corollary 4.4.1
in \cite{Bl}.\end{proof}

{\bf Remark:} From the above two lemmas we know that any
$\C[t]$-torsion-free irreducible module over the associative algebra
$\mathcal {A}$ is a restriction of a $\C[t, t^{-1}]$-torsion-free
irreducible module over the associative algebra $\mathcal {K}$ which
we have the classification as in Lemma 2. But the converse is not
true. For example $A=\C[t, t^{-1}]$ is naturally a $\mathcal {K}$
module ($\partial$ acts as derivation) which is
$\C[t,t^{-1}]$-torsion-free, while it is not an irreducible
$\mathcal {A}$ module. We also remark that it is generally very hard
to judge whether an element in $\C(t)[\partial]$ is irreducible. See Lemmas 14 and 16.

\

For any $\lambda\in\C^*$ we can define a $\mathcal {K}$-module
structure on the space $\Omega(\lambda)=\C[\partial]$, the
polynomial algebra in $\partial$, by
$$t^i\partial ^k=\lambda^i(\partial -i)^k, \partial \partial ^k=\partial^{k+1}$$
for all $k\in\Z_+, i\in\Z$. Since the action of $t$ on $\Omega(\lambda)$ has the only eigenvalue $\lambda$, we see that
different $ \lambda$ give non-isomorphic $\mathcal {K}$-modules.

It is straightforward to verify that $\Omega(\lambda)$ is a
irreducible module over the associative algebra  $\mathcal {K}$ for
any $\lambda\in\C^*$.

Actually these are the only irreducible
modules over the associative algebra $\mathcal {K}$ on which $\C[t,
t^{-1}]$ is torsion.

\begin{lemma}\label{G} Let $V$ be an  irreducible
module over the associative algebra $\mathcal {K}$ on which $\C[t,
t^{-1}]$ is torsion. Then
 $V\cong \Omega(\lambda)$
for some $\lambda\in\C^*$.
\end{lemma}
\begin{proof} Since $tt^{-1}=1$ we know that the action of $t$ (and
$t^{-1}$) is bijective on $V$. Since $\C[t, t^{-1}]$ is torsion on
$V$, there exists  $f(t)\in\C[t,t^{-1}]$ such
that $f(t)v=0$ for some nonzero $v\in V$. Moreover, by multiplying an appropriate power of $t$, we may assume
that $f(t)\in \C[t]$ with nonzero constant term. So there exists
$\lambda\in\C^*$ and a nonzero $u\in V$ such that $tu=\lambda u$. We
see that $V=\C[\partial]u$. If $\C[\partial]$ is torsion on $V$,
there exists  $f(\partial)\in\C[\partial]$ such that
$f(\partial)u=0$. So there exists $\mu\in\C$ and a nonzero $w\in V$
such that $\partial w=\mu w$. We must have $V=\C[t, t^{-1}]w$. Again
since   $\C[t, t^{-1}]$ is torsion on $V$, there exists
$g(t)\in\C[t]$ with nonzero constant term such that $g(t)w=0$ since $t$ is bijective on $V$. By repeatedly applying
$\partial$ to  $g(t)w=0$ we obtain that $w=0$, which is impossible.
So we may assume that $V=\C[\partial]$, the action of $\partial$ is
simply multiplication on the left and $t1=\lambda$. Consequently
$t^i\partial ^k=\lambda^i(\partial -i)^k$ for all $k\in\Z_+$ and
$i\in \Z$. Thus $V\cong \Omega(\lambda)$.
\end{proof}

Combining the Lemmas 2 and 3, we obtain a classification for all
irreducible modules over the associative algebra $\mathcal {K}$.

\begin{section}{Irreducible modules over the Virasoro algebra}\end{section}

We have the complete classification for irreducible modules over the
associative algebra $\mathcal {K}$ in Sect.2. These modules can be
considered as  modules over the Virasoro algebra $\Vir$ since the centerless Virasoro algebra
is a subalgebra of $\mathcal {K}$. In this section we will use
irreducible modules over the associative algebra $\mathcal {K}$ and
``the twisting technique" to construct modules over the
Virasoro algebra $\Vir$ with trivial action of the center. We will determine the necessary and sufficient conditions for such modules to be irreducible,   also
determine the necessary and sufficient conditions for two such $\Vir$
modules to be isomorphic. The twisted Heisenberg-Virasoro algebra
will be used in the proof. Let us first recall the definition of the
twisted Heisenberg-Virasoro algebra.

\

 The twisted Heisenberg-Virasoro algebra $\LL$ is the universal
central extension of the Lie algebra $\{f(t)\frac{d}{dt}+g(t)| f,g
\in \C[t,t^{-1}]\}$ of differential operators of order at most one
on the Laurent polynomial algebra $\C[t,t^{-1}]$. More precisely,
the twisted Heisenbeg-Virasoro algebra ${\LL}$ is a Lie algebra over
$\C$ with the basis
$$\{d_n,t^n ,z_1, z_2,z_3 | n \in \Z\}$$
and subject to the Lie bracket
\begin{equation}[d_n,d_m]=(m-n)d_{n+m}+\delta_{n,-m}\frac{n^3-n}{12}z_1,\end{equation}
\begin{equation}[d_n,t^m]=mt^{m+n}+\delta_{n,-m}(n^2+n)z_2, \end{equation}
\begin{equation}[t^n,t^m]=n\delta_{n,-m}z_3,\end{equation}
\begin{equation}[\LL,z_1]=[\LL,z_2]=[\LL,z_3]=0. \end{equation}

The Lie algebra $\LL$ has a Virasoro subalgebra $\Vir$ with basis
$\{d_i,z_1|i\in \Z\}$, and a Heisenberg subalgebra $\tilde{\mathcal
{H}}$ with basis $\{t^i, z_3|i\in \Z\}$.

\

Let us start with  an irreducible module $A$ over the associative
algebra ${\mathcal {K}}$.
 For any $b\in \C$, we define the action of $\LL$ on $A$ by

\begin{equation}{\label{2}}d_n  v=(t^{n}\partial +nbt^n) v,\forall \,\, n\in \Z, v\in A\end{equation}

\begin{equation}{\label{2.1}}z_k A=0,  k=1,2,3,\end{equation}
while the action of $t^n\in\LL$ is the same as $t^n\in\mathcal {K}$
on $A$. We denote by $A_b$ this module over $\mathcal {L}$. Also
$A_b$ is a module over $\Vir$. Now we determine the irreducibility
of the $\Vir$-modules $A_b$.

 \

{\bf Remark:} We do not have $\alpha\in\C[t, t^{-1}]$ in (3.5) as in
(2.3) in \cite{LGZ} since the resulting modules for different
$\alpha$ can be obtained by changing the module $A$.

\

\begin{lemma}\label{trivial}
      Let $A$ be an irreducible  module over the associative algebra $\mathcal {K}$. Then
       $A_b$ is an irreducible  module over $\LL$ for all $b\in
     \C$.\end{lemma}

\begin{proof} The statement in this Lemma is obvious.
\end{proof}

\begin{lemma}\label{key} For any $ k\in \Z$,
let \begin{equation}w_k=-\frac12 d_{k-1} d_1-\frac12 d_{k+1}
d_{-1}+d_k d_0\in U(\Vir).\end{equation} Then for all $k\in \Z$ and
$ v\in A_b$ we have $w_k v=b(b-1) t^kv$.
\end{lemma}

\begin{proof} For any $k\in \Z$ and $v\in A_b$, we compute
$$\aligned (d_{k-i}d_{i}) v=&(t^{k-i}\partial+(k-i)bt^{k-i})(t^{i}\partial+bit^i)v\\ =&
((t^k\partial^2+kbt^k\partial)+i(t^k\partial+kb^2t^k)+i^2(b-b^2)t^k)v.\endaligned$$
Taking $i=-1,0,1$ respectively, we get\newline
 $$w_k  v=b(b-1)t^k v.$$ The lemma follows.
\end{proof}

\begin{corollary}\label{cor1} Let $A$ be an irreducible module over the associative algebra $\mathcal {K}$.
If $b\in \C\setminus\{0,1\}$, then $A_b$ is an irreducible $\Vir$-module.
\end{corollary}
\begin{proof}From Lemma \ref{key}, any $\Vir$-submodule of $A_b$ is also an $\LL$-submodule. Therefore the statement follows by Lemma
\ref{trivial}.\end{proof}

The next result handles the case of $A_0$  for the irreducible module
$A$ over the associative algebra $\mathcal {K}$.

\begin{lemma}Let $A$ be an irreducible   module over the associative algebra $\mathcal {K}$.
\begin{enumerate}[$(a).$]
\item The module $A_{0}$ is irreducible   over $\Vir$ iff $A$ is not
isomorphic to the natural  module $\C[t, t^{-1}]$ over $\mathcal
{K}$. \item When $A=\C[t, t^{-1}]$ is the natural  $\mathcal {K}$
module, then $(\C[t, t^{-1}])/\C$ is an irreducible module over
$\Vir$.
\end{enumerate}
\end{lemma}

\begin{proof} (a). Let $v$ be any nonzero element in $A$.
 Denote
$$I_k(v)=\{f\in \C[t,t^{-1}]\,\,:\,\,\exists\,\,
x=f\partial^k+\sum_{i>k} f_i\partial^i\,\, \mbox{with}\,\, xv=0\}
$$ for all $k\in \Z_+$. Since $A$ is irreducible over  the associative algebra $\mathcal
{K}$ and the adjoint module $\mathcal
{K}$ is not irreducible, there exists some $k\in \Z_+$ such that $I_k(v)\ne0$. Let $m$ be the minimal non-negative
integer $k$ such that $I_k(v)\ne 0$. Note that $m$ depends on $v$. Using $\C[t,t^{-1}]xv=0$,
$\partial xv=0$ and $\partial f=f\partial +\partial(f)$ we see that
$I_{m}(v)$ is a $\partial$-invariant ideal of $\C[t,t^{-1}]$,
yielding $I_{m}(v)=\C[t,t^{-1}]$. Then
$(1+\sum_{i>0}f_i\partial^i)\partial^{m} v=0$ for some $f_i\in
\C[t,t^{-1}]$.

 \

{\bf Case 1:} $\partial^{m} v=0$.

There exists nonzero $v'\in A$ such that $\partial v'=0$. Then
$A=\C[t, t^{-1}]v'$.  We deduce that $\partial (t^kv')=kt^kv'$. It
is easy to see that $A$ is isomorphic to the natural  module $\C[t,
t^{-1}]$ over $\mathcal {K}$ which is not irreducible over $\Vir$
since it has a $1$-dimensional $\Vir$-submodule $\C$.

\

{\bf Case 2:} $\partial^{m} v\ne 0$.

Let $v'=\partial^{m} v$ and $u$ be any nonzero element in $A$. Since $A$ is an  irreducible  module over
the associative algebra $\mathcal {K}$, then $v\in \mathcal {K}v'$,
say $u=(g_0+\sum_{i>0}g_i\partial^i)v'$ for some $g_i\in \C[t,
t^{-1}]$. Recall that $(1+\sum_{i>0}f_i\partial^i)v'=0$. We see that  $u\in
g_0(1+\sum_{i>0}f_i\partial^i)v'+U(\Vir) v'=U(\Vir)v'\subset
U(\Vir)v$. Thus $A_{0}$ is an irreducible module over $\Vir$.

\

(b).  It is easy to see that $(\C[t,
t^{-1}])/\C=\span\{t^k\,:\,k\in\Z\setminus\{0\}\}$. The action is as
follows: $$d_kt^n=t^k\partial\cdot t^n=nt^{k+n}.$$ Thus $(\C[t,
t^{-1}])/\C$ is exactly the irreducible module $V'_{0,0}$ in
\cite{KR} which is irreducible over $\Vir$.
\end{proof}

The next result handles the case of $A_1$  for irreducible module
$A$ over the associative algebra $\mathcal {K}$.

\begin{lemma}Let $A$ be an irreducible   module over the associative algebra $\mathcal {K}$.
Then $\partial A_1$  is an irreducible $\Vir$-submodule of $A_1$,
and  $\partial A_1$ is isomorphic to $A_0$  as modules over $\Vir$
if $A_0$ is irreducible, or isomorphic to  $V'_{0,0}$
otherwise.\end{lemma}

\begin{proof} Since $b=1$, (3.5) becomes
\begin{equation}d_n  v=(t^{n}\partial +nt^n) v=\partial\cdot t^n\cdot v,\forall \,\, n\in \Z, v\in A.\end{equation}
We see that $\partial A_1$  is a $\Vir$-submodule of $A_1$. For any
$\partial v\in \partial A_1$, we have
$$d_n \partial v=(t^{n}\partial +nt^n)\partial v=\partial\cdot t^n\partial\cdot
v.
$$
If $A_0$ is irreducible, then from the proof of Case 1 of Lemma 7, we have $\partial v\ne 0$ for any
$0\ne v\in A$. In this case it is easy to verify
that the linear map
$$\varphi: \partial A_1\to  A_0,\,\,\, \partial v\to v, \,\,\forall\,\,v\in A$$
is a $\Vir$-module isomorphism.

If $A_0$ is not irreducible, from Lemma 7 then $A=\C[t, t^{-1}]$. It
is easy to verify that the linear map
$$\varphi: \partial A_1\to  \C[t, t^{-1}]/\C,\,\,\, \partial t^n\to t^n/n, \,\,\forall\,\,n\in \Z\setminus\{0\}
$$
is a $\Vir$-module isomorphism.
\end{proof}

Now we can summarize simplicity results as follows.

\begin{theorem}\label{iso} Suppose that $b\in \C$, and $A$ is an
 irreducible module over the associative algebra  $\mathcal
{K}$. Then $A_b$ is a irreducible  $\Vir$-module if and only if one
of the following holds
\begin{enumerate}[$(i).$]
\item
$b\ne 0$ or $1$;
\item  $b=1$ and $\partial A=A$;
\item $b=0$ and $A$ is not isomorphic to the natural $\mathcal {K}$ module $\C[t,t^{-1}]$.
\end{enumerate}\end{theorem}

\

Next we determine when two  $\Vir$-modules $A_b$ are isomorphic.

\begin{lemma}\label{iso-2} Suppose that $b, b_1\in \C$ with $b\ne 0$ or $1$, and $A$ and $B$ are
 irreducible modules over the associative algebra  $\mathcal
{K}$. Then   $A_b\cong B_{b_1}$ as $\Vir$-modules if and only if
$b=b_1$ and $A\cong B$ as $\mathcal {K}$-modules.\end{lemma}

\begin{proof} The sufficiency of the conditions is clear. Now suppose that $\mu:A_b\rightarrow B_{b_1}$
is a $\Vir$-module isomorphism. From Lemma \ref{key}, for any $v\in
A$ we have
$$b_1(b_1-1)t^k\mu(v)=w_k\mu(v)=\mu(w_kv)=b(b-1)\mu(t^kv),\,\,\forall\,\, k\in\Z.$$ Taking $k=0$ we obtain that
$b(b-1)=b_1(b_1-1)$. In particular, $b_1\notin \{0,1\}$. Noting that
\begin{equation}\label{t} \mu(t^k v)=\mu(\frac{w_k}{b(b-1)}v)=
\frac{w_k}{b_1(b_1-1)}\mu(v)=t^k\mu(v),\end{equation} we deduce that
 \begin{equation}\label{tpartial} \mu(t^k\partial^j v)=t^k\mu(\partial^j v)=t^k\mu(d_0^j v)
 =t^k d_0^j\mu(v)=t^k\partial^j \mu(v).\end{equation}
So  $A\cong B$ as $\mathcal {K}$-modules. From (\ref{tpartial}), we
have $$0=\mu(d_k v)-d_k\mu(v)=\mu(t^k\partial
v+kt^kbv)-(t^k\partial+kt^kb_1)\mu(v)$$
$$=k(b-b_1)t^k\mu(v),\,\,\,\forall\,\,v\in A\,\, k\in\Z.$$ Thus $b=b_1$. This completes the
proof.
\end{proof}

Now we consider the case $b, b_1\in\{0,1\}$.

\begin{lemma}\label{iso-3} Suppose that $A$ and $B$ are
 irreducible modules over the associative algebra  $\mathcal
{K}$. Then   $A_0\cong B_{0}$ as $\Vir$-modules if and only if
$A\cong B$ as $\mathcal {K}$-modules.\end{lemma}

\begin{proof} The sufficiency of the conditions is clear. Now suppose that $\mu:A_0\rightarrow B_{0}$
is a $\Vir$ module isomorphism. Note that $d_0=\partial$.

If $(\partial -k)v=0$ for some $k\in\Z$ and a nonzero $v\in A$, then
$\partial (t^{-k}v)=0$ where $t^{-k}v\ne0$. From Case 1 of the proof
of Lemma 7 we know that $A\cong \C[t, t^{-1}]$, the natural
$\mathcal {K}$-module. Similarly, from  $(\partial -k)\mu(v)=0$ in
$B$, we deduce that  $B\cong \C[t, t^{-1}]$. Thus $A\cong B$ as
$\mathcal {K}$-modules in this case.

Now suppose that $(\partial -k)$ is injective on both $A$ and $B$ for all $k\in \Z$.

Then for any $v\in A$, $k\in\Z$ we have $\mu(t^k\partial
v)=t^k\partial\mu( v)$. We deduce that
$$(\partial-k)\mu(t^kv)=\mu((\partial-k) t^kv)=\mu(t^k\partial v )=t^k\partial\mu( v) $$
$$=((\partial -k)t^k)\mu( v)=(\partial -k)(t^k\mu( v)),
$$
yielding that $\mu(t^k v)=t^k\mu( v)$ for all $v\in A$ and $k\in \Z$.
Therefore $A\cong B$ as $\mathcal {K}$-modules in this case also.
This completes the proof.
\end{proof}

Now we can summarize isomorphism results as follows.

\begin{theorem}\label{iso} Suppose that $b, b_1\in \C$, and $A$ and $B$ are
 irreducible modules over the associative algebra  $\mathcal
{K}$. Then $A_b\cong B_{b_1}$ as $\Vir$-modules if and only if one
of the following holds
\begin{enumerate}[$(i).$]
\item
$A\cong B$ as $\mathcal {K}$-modules,  and $b=b_1$;
\item $A\cong B$ as $\mathcal {K}$ modules, $b=1, b_1=0$ and $\partial A=A$;
\item $A\cong B$ as $\mathcal {K}$ modules, $b=0, b_1=1$ and $\partial B=B$.
\end{enumerate}\end{theorem}

\begin{section}{Old and new irreducible Virasoro modules}\end{section}

\subsection{Intermediate series modules}\label{s4.1}

Let $\alpha\in \C[t, t^{-1}]$, $b\in\C$. Take $\b=\partial -\alpha$
in Lemma 2. Then we have the irreducible  $\mathcal {K}$-module
$$A=\mathcal {K}/(\mathcal {K}\cap(\C(t)[\partial]\b))=\mathcal {K}/(\mathcal
{K}\b)$$ which has a basis $\{t^k\,:\,k\in\Z\}$ where we have
identified $t^k$ with $t^k+\mathcal {K}$ (we will continue to do
this later without mentioning). The actions of $\mathcal {K} $ are
given by
$$\partial \cdot t^n=t^n(\alpha-n), \,\,\,\,t^k\cdot
t^n=t^{k+n},\forall\,\, k,n\in\Z.
$$
Using (3.5) we obtain Vir-modules $A_{\alpha, b}=\C[t, t^{-1}]$ with
the action:
$$d_k\cdot t^n=(\alpha+n+kb)t^{k+n},\forall\,\,
k,n\in\Z.
$$
These modules $A_{\alpha, b}$ are exactly the ones introduced and
studied in Section 4 of the paper \cite{LGZ}. When $\alpha\in\C$
these modules $A_{\alpha, b}$ are the intermediate series modules
$V_{\alpha, b}$ in \cite{KR}.

\subsection{Fraction modules}\label{s4.3}

Let $b\in \C$, ${\bf \a}=(\alpha_0,\alpha_1,\alpha_2,...,\alpha_n), \newline {\bf{a}}=(a_0, a_1,
a_2,...,a_n)\in \C^{n+1}$ with $a_0=0$ and $a_i\ne a_j$ for all $i\ne j$. In Lemma 2 take $$\b=\frac{d}{d t}
-\sum_{i=0}^n\frac{\alpha_i}{t-a_i}.$$ Then we have the irreducible
$\mathcal {K}$-module
$$A=\mathcal {K}/(\mathcal {K}\cap(\C(t)[\partial]\b))\subset \C[t, (t-a_i)^{-1}\,\,|\,\,
i=0,1,2,...,n]. $$ The actions of $\mathcal {K} $ are given by
$$\frac{d}{d t} \cdot f(t)=\frac{d}{d t} (f(t))+f(t)\sum_{i=0}^n\frac{\alpha_i}{t-a_i}, $$
$$t^k\cdot f(t)=t^{k}f(t),\forall\,\, k\in\Z, f\in A.
$$
It is not very hard to verify  that the Virasoro modules $A_b$ is a submodule of the fraction module
 $V({\bf{a}}, {\bf \alpha}-b\epsilon_0,b)$  introduced and studied in
Section 2 of the paper \cite{GLZ}, where $\epsilon_0=(1,0,\ldots,0)\in \C^{n+1}$.  And we have  $A_b=V({\bf{a}}, {\bf \alpha}-b\epsilon_0,b)$ if $V({\bf{a}}, {\bf{\alpha}}-b\epsilon_0,b)$ is simple.

\
\subsection{Virasoro modules $\Omega(\lambda, b)$}\label{s4.1}

Let $\lambda\in \C^*$ and  $b\in\C$.   Then we have the irreducible
$\mathcal {K}$-module $\Omega(\lambda)$ which has a basis
$\{\partial^k\,:\,k\in\Z_+\}$. The actions of $\mathcal {K} $ are
given by
$$t^i \cdot \partial ^k=\lambda^i(\partial -i)^k, \,\,\,\partial  \cdot \partial
^k=\partial^{k+1},\forall\,\, k\in\Z_+,i\in\Z.
$$
Using (3.5) we obtain $\Vir$-modules $\Omega(\lambda, b)=\C[\partial]$
with the action: \begin{equation} d_n\cdot
\partial^k=\lambda^n(\partial+n(b-1))(\partial-n)^{k},\forall\,\,
k\in \Z_+,n\in\Z.
\end{equation}
From Theorem 10 we know that the $\Vir$-modules $\Omega(\lambda, b)$
are irreducible if $b\ne 1$. It is easy to see that $\Omega(\lambda,
1)$ has an irreducible submodule $\partial \C[\partial]$ which is
isomorphic to $\Omega(\lambda, 0)$.

The Virasoro modules $\Omega(\lambda, b)$ for $\lambda\in \C^*$ and
$b\in\C$ are very similar to the highest-weight-like Virasoro
modules $V(\xi, \lambda)$ defined in Section 3 of \cite{GLZ}.
Actually they are isomorphic.

\begin{lemma}\label{iso-5} Suppose that $\lambda\in\C^*, b\in \C$. Then the Virasoro modules
$\Omega(\lambda, b)$ and $V(\lambda, b-1)$ are
isomorphic.\end{lemma}

\begin{proof} Denote by $\partial'$ the operator $\partial
=\frac d{dt}$ in \cite{GLZ}. We know that $V(\lambda,
b-1)=\C[d_{-1}]v$ where $\C[d_{-1}]$ is the polynomial algebra in
$d_{-1}$ with the properties:
$$(d_0-\lambda d_{-1})v=\lambda v,$$
$$(d_1-2\lambda d_0+\lambda^2 d_{-1})v=0,$$
$$(d_2-3\lambda d_1+3\lambda^2 d_{0}-\lambda^3 d_{-1})v=0,$$
$$(d_{k-1}-\lambda^k
d_{-1}-k\lambda^{k-1}(b-1))v=0,\,\,\forall\,\,k\in\Z,$$ and these
properties characterize the Virasoro module $V(\lambda, b-1)$. It is
straightforward to verify that all the above properties are
satisfied by the Virasoro module $\Omega(\lambda, b-1)$ with $v$
replaced by $1$.
\end{proof}

\

We remark that in a recent preprint \cite{TZ}, it was proved that
the tensor product of $\Omega(\lambda, b)$ ($b\ne1$) with an
irreducible highest weight module or with an irreducible module
defined in \cite{MZ2} is also an irreducible Virasoro module.

\subsection{Degree two modules}\label{s4.4}

As we mentioned before,  for a given element $f(t, \partial)$ in
$\C(t)[\partial]$ (which is a left principal ideal domain) it is
generally very hard to know whether $f(t,
\partial)$ is irreducible in $\C(t)[\partial]$ or not (see the next
three lemmas). We first construct some degree two irreducible
elements in $\C(t)[\partial]$.

\begin{lemma}\label{exam-4} Suppose that $f(t)\in \C[t, t^{-1}]$.
Then $\partial^2 -f(t)$ is irreducible in $\C(t)[\partial]$ iff
$f(t)$ is not of the form
\begin{equation}h(t)^2-\partial(h(t))-2\sum_{i=1}^n\frac{a_i(h(t)-h(a_i))}{t-a_i}\end{equation} for any
$h(t)\in\C[t, t^{-1}]$ satisfying
$$h(a_i)=\sum_{j\ne i}\frac{a_j}{a_i-a_j}-\frac{1}{2},\,\, \forall\,\,
i=1,2,...,n,$$ where $n\in\Z_+$, and $a_1,a_2,...,a_n\in \C^*$ are
pairwise distinct.\end{lemma}

\begin{proof}
Suppose $\partial^2 -f(t)$ is reducible in $\C(t)[\partial]$. Then
there exists $g_1,g_2\in\C(t)$ such that $\partial^2
-f(t)=(\partial-g_1)(\partial-g_2)$, to give $\partial^2
-f(t)=\partial^2 -(g_1+g_2)\partial +g_1g_2-\partial(g_2)$. We see
that $g_2=-g_1$ and $f=g_1^2-\partial(g_1)\in\C[t,t^{-1}]$.
Write
$$g_1=\sum_{i=1}^n\sum_{j=1}^{l_i}
\frac{c_{i,j}}{(t-a_i)^j}+h(t),$$ where  $a_i\in \C^*$ are pairwise
distinct, $c_{i,j}\in \C$ and  $h(t)\in \C[t,t^{-1}]$. We may assume
that  $l_i\ge 1$ and $c_{i,l_i}\ne 0$ for each $i=1,2,\ldots,n$.  By computing the coefficient of $(t-a_i)^{-2l_i}$ in $g_1^2-\partial(g_1)\in \C[t,t^{-1}]$, we have $l_i=1$ and $c_{i,1}=-a_i$. So
$$g_1=\sum_{i=1}^n \frac{-a_i}{t-a_i}+h(t).$$
Consequently,  $$\aligned f=&g_1^2-\partial(g_1)=h(t)^2+\sum_{i> j}
\frac{2a_ia_j}{(t-a_i)(t-a_j)}\\
&+2h(t)\sum _{i=1}^n\frac{-a_i}{t-a_i}-\partial(h(t))-\sum_{i=1}^n
\frac{a_i}{t-a_i}\in \C[t,t^{-1}].\endaligned$$ That is
$$f-h(t)^2+\partial(h(t))=\sum_{i>j}
\frac{2a_ia_j}{(t-a_i)(t-a_j)}+2(h(t)+\frac{1}{2})\sum_{i=1}^n
\frac{-a_i}{t-a_i}$$ $$=\sum_{i>j}
\frac{2a_ia_j}{(a_i-a_j)}(\frac{1}{t-a_i}-\frac{1}{t-a_j})+2(h(t)+\frac{1}{2})\sum_{i=1}^n
\frac{-a_i}{t-a_i}$$
$$=\sum_{i=1}^n\Big((-(h(t)+1/2)+
\sum_{j\ne i}\frac{a_j}{a_i-a_j}\Big)\frac{2a_i}{t-a_i} \in
\C[t,t^{-1}].$$
 Therefore $h(t)$ satisfies the condition
$$h(a_i)=\sum_{j\ne i}\frac{a_j}{a_i-a_j}-\frac{1}{2},\,\, \forall\,\,
i=1,2,...,n.$$ We simplify
$$\sum_{i=1}^n
\Big(-(h(t)+1/2)+\sum_{j\ne
i}\frac{a_j}{a_i-a_j}\Big)\frac{2a_i}{t-a_i}=-2\sum_{i=1}^n\frac{a_i(h(t)-h(a_i))}{t-a_i}.$$
We see that
$$f=h(t)^2-\partial(h(t))-2\sum_{i=1}^n\frac{a_i(h(t)-h(a_i))}{t-a_i}.$$

The converse is clear from the above arguments.\end{proof}

Remark that the last term in (4.2) will disappear if $n=0$.

\

{\bf Example 1.} Take $n=1, a_1=1$, $h(t)=t-3/2$. Using Lemma 14 we
obtain that  $f(t)=t^2-4t+\frac{1}{4}$, and $\partial^2-f(t)$ is
reducible in $\C(t)[\partial]$. But it is not hard to verify that
$\partial^2-f(t)$ is irreducible in $\mathcal {K}$.

 {\bf Example 2.} From Lemma 14 we know that,
 if $f(t)\in\C[t,t^{-1}]$ with odd positive highest degree (or odd
negative lowest degree) then  $\partial^2-f(t)$ is irreducible in
$\C(t)[\partial]$. Certainly, in this case $\partial^2-f(t)$ is
automatically  irreducible in $\mathcal {K}$.

\

Now let $f(t)\in\C[t,t^{-1}]$ be such that $\partial^2-f(t)$ is
irreducible in $\C(t)[\partial]$. Take $\b=\partial^2-f(t)$ in Lemma
2. Then we have the irreducible $\mathcal {K}$-module
\begin{equation}A=\mathcal {K}/(\mathcal {K}\cap(\C(t)[\partial]\b))=\mathcal
{K}/(\mathcal {K}\b)\end{equation} which has a basis $\{t^k,
t^k\partial\,:\,k\in\Z\}$. (Remark that the second equality of the
above equation is the reason why we have assumed that $f\in
\C[t,t^{-1}]$ instead of in $\C(t)$ in Lemma 14).  The actions of
$\mathcal {K} $ on $A$ are given by
$$t^k\cdot t^n=t^{k+n},\,\,\,\,t^k\cdot(t^n\partial )=t^{k+n}\partial,$$
$$
\partial \cdot t^n=t^n(\partial+n), \,\,\,\,
\partial \cdot (t^n\partial)=t^{n}(f(t)+n\partial),\forall\,\, k,n\in\Z.
$$
Using (3.5), for any $b\in\C\setminus\{1\}$  we obtain irreducible
$\Vir$-modules $A_{b}=\C[t, t^{-1}]\oplus \C[t, t^{-1}]\partial$ with
the action:
$$d_k\cdot t^n=t^{k+n}(n+kb+\partial),
$$
$$d_k\cdot (t^n\partial)=t^{k+n}(f(t)+kb+n\partial),\forall\,\,
k,n\in\Z.
$$

As far as we know, these modules $A_{b}$ just constructed are new
irreducible $\Vir$-modules (also see Theorem 17).

\

The next result combining with Lemma 14 gives more irreducible
elements in $\C(t)[\partial]$ and more irreducible Virasoro modules.

\begin{lemma} Suppose that $f_1(t),f_2(t)\in \C[t,t^{-1}]$.
Then $\partial^2+2f_1\partial +f_2(t)$ is irreducible in
$\C(t)[\partial]$ iff $\partial^2-(\partial(f_1)+f_1^2-f_2)$ is
irreducible in $\C(t)[\partial]$.\end{lemma}

\begin{proof} It is easy to see that the linear map $\tau: \C(t)[\partial]\to \C(t)[\partial]$ defined by $
(\sum_i g_i(t)\partial^i)=\sum_i g_i(t)(\partial-f_1)^i$ is an
automorphism of $\C(t)[\partial]$. We write $\partial^2+2f_1\partial
+f_2(t)= (\partial+f_1(t))^2-(\partial(f_1)+f_1^2-f_2)$. Hence the
lemma follows from the fact that $\tau(x)$ and $x$ ($x\in
\C(t)[\partial]$) have the same irreducibility in $\C(t)[\partial]$.
\end{proof}

\subsection{Degree $n$ modules}\label{s4.4}

We first construct some degree $n$ irreducible elements in
$\C(t)[\partial]$.

\begin{lemma}\label{exam-4} For any nonconstant polynomial $f(t)\in \C[t]$,
the element $f(\frac{{d}}{{d} t})-t\in \C[t,t^{-1}][\partial]$ is
irreducible in $\C(t)[\partial]$.\end{lemma}

\begin{proof} Let $J$ be the left ideal generated by $f(\frac{\mbox{d}}{\mbox{d} t})-t$
in the associative algebra $\C[t][\frac{\mbox{d}}{\mbox{d} t}]$.
Then the $\C[t][\frac{\mbox{d}}{\mbox{d} t}]$-module
$\C[t][\frac{\mbox{d}}{\mbox{d} t}]/J$ has a basis
$\{(\frac{\mbox{d}}{\mbox{d} t})^i|i\in \Z_+\}$ with the action
$$\frac{\mbox{d}}{\mbox{d} t}\cdot (\frac{\mbox{d}}{\mbox{d} t})^i=(\frac{\mbox{d}}{\mbox{d} t})^{i+1},$$
$$t \cdot (\frac{\mbox{d}}{\mbox{d} t})^i=-i(\frac{\mbox{d}}{\mbox{d} t})^{i-1}
+f(\frac{\mbox{d}}{\mbox{d} t})(\frac{\mbox{d}}{\mbox{d} t})^i.$$

For any nonzero submodule $V$ of $\C[t][\frac{\mbox{d}}{\mbox{d}
t}]/J$, it is easy to see that $V$ is an ideal of the polynomial
algebra $\C[\frac{\mbox{d}}{\mbox{d} t}]$. For any
$g(\frac{\mbox{d}}{\mbox{d} t})\in V$, from $-t \cdot
g(\frac{\mbox{d}}{\mbox{d} t})\in V$ we see that
$$-t
\cdot g(\frac{\mbox{d}}{\mbox{d} t})=[-t, g(\frac{\mbox{d}}{\mbox{d}
t})]- g(\frac{\mbox{d}}{\mbox{d} t})t = g'(\frac{\mbox{d}}{\mbox{d}
t})-f(\frac{\mbox{d}}{\mbox{d} t}) g(\frac{\mbox{d}}{\mbox{d} t})\in
V,$$ yielding $g'(\frac{\mbox{d}}{\mbox{d} t})\in V$. We deduce that
$1\in V$ and consequently $V= \C[t][\frac{\mbox{d}}{\mbox{d} t}]/J$.
Hence $\C[t][\frac{\mbox{d}}{\mbox{d} t}]/J$ is a simple module over
the associative algebra  $\C[t][\frac{\mbox{d}}{\mbox{d} t}]$. Note
that $\C[t][\frac{\mbox{d}}{\mbox{d} t}]/J$ is $\C[t]$ torsion-free.
Therefore $\C(t)[\frac{\mbox{d}}{\mbox{d}
t}](f(\frac{\mbox{d}}{\mbox{d} t})-t)$ is a maximal left ideal of
$\C(t)[\frac{\mbox{d}}{\mbox{d} t}]$. Thus
$f(\frac{\mbox{d}}{\mbox{d} t})-t$ is irreducible in
$\C(t)[\partial]$.
\end{proof}

For any $n\in\N$, take $\b=(\frac{\mbox{d}}{\mbox{d} t})^n-t$ in
Lemma 2. Then we have the irreducible $\mathcal {K}$-module
$$A=\mathcal {K}/(\mathcal {K}\cap(\C(t)[\partial]\b))=\mathcal {K}/(\mathcal
{K}\b)$$ which has a basis $\{t^k(\frac{\mbox{d}}{\mbox{d}
t})^m\,:\,k\in\Z, m=0,1,...,n-1\}$. The actions of $\mathcal
{K}=\C[t,t^{-1}][(\frac{\mbox{d}}{\mbox{d} t})] $ are given by
$$t^k\cdot(t^r(\frac{\mbox{d}}{\mbox{d} t})^m )=t^{k+r}(\frac{\mbox{d}}{\mbox{d} t})^m,\forall\,\, k,r\in\Z, 0\le m\le n-1,$$
$$
(\frac{\mbox{d}}{\mbox{d} t}) \cdot (t^r(\frac{\mbox{d}}{\mbox{d}
t})^s) =rt^{r-1}(\frac{\mbox{d}}{\mbox{d}
t})^s+t^r(\frac{\mbox{d}}{\mbox{d} t})^{s+1},\forall\,\, r\in\Z,
0\le s< n-1,
$$
$$
(\frac{\mbox{d}}{\mbox{d} t}) \cdot (t^r(\frac{\mbox{d}}{\mbox{d}
t})^{n-1}) =rt^{r-1}(\frac{\mbox{d}}{\mbox{d}
t})^{n-1}+t^{r+1},\forall\,\, r\in\Z.
$$
Using (3.5), for any $b\in\C\setminus\{1\}$  we obtain irreducible
$\Vir$-modules $A_{b}=\C[t, t^{-1}] (\sum_{i=0}^{n-1}\C
(\frac{\mbox{d}}{\mbox{d} t})^i) $ with the action:
$$
d_k \cdot (t^r(\frac{\mbox{d}}{\mbox{d} t})^s)
=(rt^{k+r}+bkt^{k+r+1})(\frac{\mbox{d}}{\mbox{d}
t})^{s}+t^{k+r+1}(\frac{\mbox{d}}{\mbox{d} t})^{s+1},\forall\,\,
k,r\in\Z,
$$
$$
d_k\cdot (t^r(\frac{\mbox{d}}{\mbox{d} t})^{n-1})
=(rt^{k+r}+bkt^{k+r+1})(\frac{\mbox{d}}{\mbox{d}
t})^{n-1}+t^{k+r+2},\forall\,\, k,r\in\Z,
$$
where $0\le s< n-1$.

For different $\beta$ (different $n$), when we consider the
$\mathcal {K}$-modules $A$ as $\C[t, t^{-1}]$-modules they are not
isomorphic since they are free $\C[t,t^{-1}]$-modules of rank $n$.
From Theorem 12 we know that we have obtained many non-isomorphic
irreducible Virasoro modules in this way. As far as we know, these
modules $A_{b}$ are new irreducible Vir modules (also see Theorem
17).

\

We would like to conclude  this paper by comparing the Virasoro modules
$A_b$ with other known irreducible Virasoro modules. It is not hard
to see that Virasoro weight modules of the form $A_b$ are the
intermediate series modules $V_{\alpha, \b}$ in \cite{KR} for which
$A$ is $\C[\partial]$-torsion. If $A$ is a $\C[t,t^{-1}]$-torsion
irreducible module over the associative algebra  $\mathcal {K}$,
then $A_b$ are the modules $\Omega(\lambda,b)$ in Sect.4.3.

All other known non-weight Virasoro modules are from \cite{LGZ, GLZ,
MZ2, MW}. We have already compared with those in \cite{LGZ, GLZ}.
Let us recall those modules in \cite{MZ2}. Let
$\Vir_+=\span\{d_i\,|\,i\in\Z_+\}$. Given $N\in
{\Vir}_+\text{-}\mathrm{mod}$ and $\theta\in\mathbb{C}$, consider
the corresponding induced module
$\mathrm{Ind}(N):=U(\Vir)\otimes_{U(\Vir_+)}N$ and denote by
$\mathrm{Ind}_{\theta}(N)$ the module
$\mathrm{Ind}(N)/({c}-\theta)\mathrm{Ind}(N)$.

\begin{theorem}\label{Comp} Suppose that $b\in \C\setminus\{1\}$, and $A$ is an
 irreducible module over the associative algebra  $\mathcal
{K}$ which is $\C[t,t^{-1}]$-torsion-free and
$\C[\partial]$-torsion-free. Then $A_b $ is not isomorphic to
$\mathrm{Ind}_{\theta}(N)$ for any irreducible $N\in
{\Vir}_+\text{-}\mathrm{mod}$, or the modules Ind$_{\theta,
z}(\C_{\bf m})$ defined in \cite{MW} for any $\theta, m_2,m_3,$ $
m_4\in\C$ and $z\in\C^*$.\end{theorem}

\begin{proof} It was proved
that $V=\mathrm{Ind}_{\theta}(N)$ is irreducible over $\Vir$ if
$N\in {\Vir}_+\text{-}\mathrm{mod}$ is irreducible and  $d_kN=0$ for
all sufficiently large $k$. For any $v\in V$ we have $d_kv=0$ for
all sufficiently large $k$. But this property cannot be shared by
$A_b$ because of (3.5) and the fact that $A$ is
$\C[t,t^{-1}]$-torsion-free and $\C[\partial]$-torsion-free. The
statement in the theorem follows for this case.

Now we compare our module $A_b $ with the irreducible Virasoro
modules $W=$Ind$_{\theta, z}(\C_{\bf m})$ defined in \cite{MW},
where $\theta, m_2,m_3,$ $ m_4\in\C$ and $z\in\C^*$ satisfying the
conditions \begin{equation}zm_3\ne m_4, 2zm_2\ne m_3,3zm_3\ne
2m_4,z^2m_2+m_4 \ne2z m_3.\end{equation} From the definition of $W$,
there exists a nonzero vector $v\in W$ such that
$$(d_2-zd_1-m_2)v=0, (d_3-z^2d_1-m_3)v=0,(d_4-z^3d_1-m_4)v=0
$$ yielding that $$(d_3-zd_2+zm_2-m_3)v=0, (d_4-zd_3+zm_3-m_4)v=0.$$
If $A_b\simeq W$, there exists nonzero $u\in A_b$ such that
$$(d_3-zd_2+zm_2-m_3)u=0, (d_4-zd_3+zm_3-m_4)u=0,$$
i.e.,
$$t^2(t-z)\partial +bt^2(3t-2z)+zm_2-m_3)u=0,$$
$$t^3(t-z)\partial +bt^3(4t-3z)+zm_3-m_4)u=0.$$
We obtain that
$$(bt^3(t-z)-t(zm_2-m_3)+(zm_3-m_4))u=0.$$
Since $\C[t]$ is torsion-free on $A_b$, we deduce that $zm_3=m_4$
which is a contradiction to (4.4). Therefore $A_b\not\simeq W$.
\end{proof}

\noindent {\bf Acknowledgement.} K. Z. is partially
supported by NSF of China  (Grant 11271109), NSERC and University Research Professor grant at Wilfrid Laurier University. R.L. is partially supported by NSF of China
(Grant 11371134) and Jiangsu Government Scholarship for Overseas Studies (JS-2013-313). We thank
Hongjia Chen, and Xiangqian Guo for their comments on the original
version of the paper and the referee for many helpful suggestions.

\vspace{.3cm}

\noindent  R.L.: Department of Mathematics, Soochow university,
Suzhou 215006, Jiangsu, P. R. China.
 Email: rencail@amss.ac.cn

\vspace{0.2cm} \noindent K.Z.: Department of Mathematics, Wilfrid
Laurier University, Waterloo, ON, Canada N2L 3C5,  and College of
Mathematics and Information Science, Hebei Normal (Teachers)
University, Shijiazhuang, Hebei, 050016 P. R. China. Email:
kzhao@wlu.ca
\end{document}